\documentclass[11pt,twoside]{article}

\usepackage{amssymb}
\usepackage{amsmath}
\usepackage{amsthm}
\usepackage{color}
\usepackage{colortbl,dcolumn}

\allowdisplaybreaks

\def\rr{{\mathbb R}}

\def\rn{{{\rr}^n}}

\def\zz{{\mathbb Z}}
\def\nn{{\mathbb N}}

\def\cc{{\mathbb C}}

\def\cx{{\mathcal X}}

\def\fz{\infty}
\def\az{\alpha}
\def\supp{{\mathop\mathrm{\,supp\,}}}

\def\loc{{\mathop\mathrm{\,loc\,}}}

\def\lz{\lambda}
\def\dz{\delta}

\def\ez{\epsilon}

\def\kz{\kappa}
\def\bz{\beta}
\def\fai{\varphi}
\def\gz{{\gamma}}
\def\oz{{\omega}}

\def\wz{\widetilde}

\def\hs{\hspace{0.3cm}}

\def\ls{\lesssim}

\def\rbmo{{\mathop\mathrm{RBMO}}}
\def\rbmop{{{\mathop\mathrm{RBMO}^p_\eta}(\mu)}}

\def\lp{{L^p(\mu)}}

\def\dint{\displaystyle\int}

\def\gfz{\genfrac{}{}{0pt}{}}

\def\dfrac{\displaystyle\frac}

\def\r{\right}
\def\lf{\left}

\newtheorem{thm}{Theorem}[section]
\newtheorem{lem}{Lemma}[section]
\newtheorem{prop}{Proposition}[section]
\newtheorem{rem}{Remark}[section]

\newtheorem{defn}{Definition}[section]

\numberwithin{equation}{section}

\begin{document}

\arraycolsep=1pt

\title{\Large\bf An Interpolation Theorem for Sublinear Operators
on Non-homogeneous Metric Measure Spaces\footnotetext{\hspace{-0.35cm} 2010 {\it Mathematics Subject
Classification}. {Primary 42B35; Secondary 42B25, 47B38.}
\endgraf{\it Key words and phrases.} upper doubling, geometrically
doubling, ${\mathop\mathrm{RBMO}}(\mu)$, sublinear, interpolation.}}
\author{Haibo Lin and Dongyong Yang\,\footnote{Corresponding
author.}}
\date{ }
\maketitle

\begin{center}
\begin{minipage}{13.6cm}\small
{\noindent{\bf Abstract.} Let $({\mathcal X}, d, \mu)$ be a metric
measure space and satisfy the so-called upper doubling condition and
the geometrically doubling condition. In this paper, the authors
establish an interpolation result that a sublinear operator which is
bounded from the Hardy space $H^1(\mu)$ to $L^{1,\,\infty}(\mu)$ and from
$L^\infty(\mu)$ to the BMO-type space ${\mathop\mathrm{RBMO}}(\mu)$ is also bounded on
$L^p(\mu)$ for all $p\in(1,\,\fz)$. This extension is not completely
straightforward and improves the existing result.}
\end{minipage}
\end{center}

\section{Introduction\label{s1}}

\hskip\parindent Spaces of homogeneous type were introduced by
Coifman and Weiss \cite{cw71} as a general framework in which many
results from real and harmonic analysis on Euclidean spaces have
their natural extensions; see, for example, \cite{cw77,hk,he}.
Recall that a metric space $(\cx,\,d)$ equipped with a nonnegative
Borel measure $\mu$ is called a {\it space of homogeneous type} if
$(\cx,\,d,\,\mu)$ satisfies the following {\it measure doubling
condition} that there exists a positive constant $C_\mu$, depending
on $\mu$, such that for any ball
$B(x,\,r)\equiv\{y\in\cx:\,d(x,\,y)<r\}$ with $x\in\cx$ and
$r\in(0,\,\fz)$,
\begin{equation}\label{e1.1}
0<\mu(B(x,\,2r))\le C_\mu\mu(B(x,\,r)).
\end{equation}

The measure doubling condition \eqref{e1.1} plays a key role in the
classical theory of Calder\'on-Zygmund operators. However, recently,
many classical results concerning the theory of Calder\'on-Zygmund
operators and function spaces have been proved still valid if the
doubling condition is replaced by a less demanding condition such as
the polynomial growth condition; see, for example
\cite{mmno,t01ma,t01am,ntv,t03} and the references therein. In
particular, let $\mu$ be a non-negative Radon measure on $\rn$ which
only satisfies the {\it polynomial growth condition} that there
exist positive constants $C$ and $\kz\in(0, n]$ such that for all
$x\in\rn$ and $r\in(0,\,\fz)$,
\begin{equation}\label{e1.2}
\mu(\{y\in\rn:\,|x-y|<r\})\le C r^\kz.
\end{equation}
Such a measure does not need to satisfy the doubling condition
\eqref{e1.1}. We mention that the analysis with non-doubling
measures played a striking role in solving the long-standing open
Painlev\'e's problem by Tolsa in \cite{t03}.

Because measures satisfying \eqref{e1.2} are only different, not
more general than measures satisfying \eqref{e1.1}, the
Calder\'on-Zygmund theory with non-doubling measures is not in all
respects a generalization of the corresponding theory of spaces of
homogeneous type. To include the spaces of homogeneous type and
Euclidean spaces with a non-negative Radon measure satisfying a
polynomial growth condition, Hyt\"onen \cite{h10} introduced a new
class of metric measure spaces which satisfy the so-called upper
doubling condition and the geometrically doubling condition (see,
respectively, Definitions \ref{d2.1} and \ref{d2.2} below), and a
notion of the regularized BMO space, namely, $\rbmo(\mu)$ (see
Definition \ref{d2.4} below). Since then, more and more papers focus
on this new class of spaces; see, for example
\cite{hm,ly11,hyy10,ad10,hlyy,hmy,lyy}.

Let $(\cx,\,d,\,\mu)$ be a metric measure space satisfying the upper
doubling condition and the geometrically doubling condition. In
\cite{hyy10}, the atomic Hardy space $H^1(\mu)$ (see Definition
\ref{d2.5} below) was studied and the duality between $H^1(\mu)$ and
$\rbmo(\mu)$ of Hyt\"onen was established. Some of results in
\cite{hyy10} were also independently obtained by Anh and Duong
\cite{ad10} via different approaches. Moreover, Anh and Duong
\cite[Theorem 6.4]{ad10} established an interpolation result that a
linear operator which is bounded from $H^1(\mu)$ to $L^1(\mu)$ and
from $L^\infty(\mu)$ to ${\mathop\mathrm{RBMO}}(\mu)$ is also
bounded on $L^p(\mu)$ for all $p\in(1,\,\fz)$.
The purpose of this paper is to generalize and improve the
interpolation result for
 linear operators in \cite{ad10} to  sublinear operators in the current setting
$(\cx,\,d,\,\mu)$, which is stated as follows.

\begin{thm}\label{t1.1} Let $T$ be a sublinear operator that is
bounded from $L^\fz(\mu)$ to $\rbmo(\mu)$ and from $H^1(\mu)$ to
$L^{1,\,\fz}(\mu)$. Then $T$ extends boundedly to $L^p(\mu)$ for
every $p\in(1,\,\fz)$.
\end{thm}

In Section \ref{s2}, we collect preliminaries we need. In Section
\ref{s3}, for $r\in(0, 1)$, we first show that the maximal function $M^\sharp_r(f)$, which is a
variant of the sharp maximal function $M^\sharp(f)$ in \cite{ad10}, is
bounded from $\rbmo(\mu)$ to $L^\fz(\mu)$, then we
establish a weak type estimate between the doubling maximal function $N(f)$ and $M^\sharp(f)$,
and we also establish a weak type estimate for $N_r(f)$ with $r\in(0, 1)$, a variant of $N(f)$.
Using these results we establish Theorem \ref{t1.1}.
We remark that the method for the proof of Theorem \ref{t1.1} is
different from that of \cite[Theorem 6.4]{ad10}.
Precisely, in the proof of \cite[Theorem 6.4]{ad10},
the fact that the composite operator $M^\sharp\circ T$ of the sharp maximal
function $M^\sharp$ and a linear operator $T$ is a sublinear operator was used.
However, as far as we know, when $T$ is sublinear,
whether the composite operator $M^\sharp\circ T$ is a sublinear operator is unclear
and so the proof of \cite[Theorem 6.4]{ad10} is not available.

Throughout this paper, we denote by $C$ a positive constant which is
independent of the main parameters involved but may vary from line
to line. The subscripts of a constant indicate the parameters it
depends on. The notation $f\ls g$ means that there exists a constant
$C>0$ such that $f\le Cg$. Also, for a $\mu$-measurable set $E$,
$\chi_E$ denotes its characteristic function.

\section{Preliminaries}\label{s2}

\hskip\parindent
In this section, we will recall some necessary
notions and notation and the Calder\'on-Zygmund decomposition
which was established in \cite{ad10}. We begin with the definition of
upper doubling space in \cite{h10}.

\begin{defn}\label{d2.1}\rm
A metric measure space $(\cx,\,d,\,\mu)$ is called {\it upper
doubling} if $\mu$ is a Borel measure on $\cx$ and there exists a
dominating function $\lz:\,\cx\times(0,\,\fz)\rightarrow (0,\,\fz)$
and a positive constant $C_\lz$ such that for each $x\in\cx$,
$r\rightarrow \lz(x,\,r)$ is non-decreasing, and for all  $x\in\cx$
and $r\in(0,\,\fz)$,
\begin{equation*}
\mu(B(x,\,r))\le\lz(x,\,r)\le C_\lz\lz(x,\,r/2).
\end{equation*}
\end{defn}

\begin{rem}\label{r1.1}\rm
(i) Obviously, a space of homogeneous type is a special case of the
upper doubling spaces, where one can take the dominating function
$\lz(x,\,r)\equiv\mu(B(x,\,r))$. Moreover, let $\mu$ be a
non-negative Radon measure on $\rn$ which only satisfies the
polynomial growth condition \eqref{e1.2}. By taking $\lz(x,\,r)\equiv Cr^\kz$, we
see that $(\rn,\,|\cdot|,\,\mu)$ is also an upper doubling measure
space.

(ii) It was proved in \cite{hyy10} that there exists a dominating
function $\wz\lz$ related to $\lz$ satisfying the property that
there exists a positive constant $C_{\wz\lz}$ such that
$\wz\lz\le\lz$, $C_{\wz\lz}\le C_{\lz}$, and for all $x,\,y\in\cx$
with $d(x,\,y)\le r$,
\begin{equation}\label{e2.1}
\wz\lz(x,\,r)\le C_{\wz\lz}\wz\lz(y,\,r).
\end{equation}
Based on this, in this paper, we {\it always assume} that the
dominating function $\lz$ also satisfies \eqref{e2.1}.

\end{rem}

Throughout the whole paper, we also {\it always assume} that the
underlying metric space $(\cx,\,d)$ satisfies the following
geometrically doubling condition introduced in \cite{h10}.

\begin{defn}\label{d2.2}\rm
A metric space $(\cx,\,d)$ is called {\it geometrically doubling} if
there exists some $N_0\in\nn^+\equiv\{1,\,2,\,\cdots\}$ such that
for any ball $B(x,\,r)\subset\cx$, there exists a finite ball
covering $\{B(x_i,\,r/2)\}_i$ of $B(x,\,r)$ such that the
cardinality of this covering is at most $N_0$.
\end{defn}

The following coefficients $\dz(B,\,S)$ for all balls $B$ and $S$
were introduced in \cite{h10} as analogues of Tolsa's numbers
$K_{Q,\,R}$ in \cite{t01ma}; see also \cite{hyy10}.

\begin{defn}\label{d2.3}\rm
For all balls $B\subset S$, let
$$
\dz(B,\,S)\equiv\int_{(2S)\setminus
B}\frac{d\mu(x)}{\lz(c_B,d(x,\,c_B))}.
$$
where and in that follows, for a ball $B\equiv B(c_B,\,r_B)$ and
$\rho\in(0,\,\fz)$, $\rho B\equiv B(c_B,\,\rho r_B)$.
\end{defn}

In what follows, for each $p\in(0, \fz)$, $L^p_\loc(\mu)$ denotes
the {\it set of all functions $f$ such that $|f|^p$ is $\mu$-locally
integrable}.

\begin{defn}\label{d2.4}\rm
Let $\eta\in(1, \fz)$ and $p\in(0, \fz)$. A function $f\in L_{\rm
loc}^p(\mu)$ is said to be in the {\it space $\rbmop$} if there
exists a non-negative constant $C$ and a complex number $f_B$ for
any ball $B$ such that for all balls $B$,
\begin{equation*}
\frac{1}{\mu(\eta B)}\int_B|f(y)-f_B|^p\,d\mu(y)\le C^p
\end{equation*}
and that for all balls $B\subset S$,
\begin{equation*}
|f_B-f_S|\le C[1+\dz(B,\,S)].
\end{equation*}
Moreover, the {\it $\rbmop$ norm} of $f$ is defined to be the
minimal constant $C$ as above and denoted by $\|f\|_\rbmop$.
\end{defn}

When $p=1$, we write ${\rm RBMO}^1_\eta(\mu)$ simply by $\rbmo(\mu)$, which
was introduced by Hyt\"onen in \cite{h10}. Moreover, the spaces $\rbmop$
and $\rbmo(\mu)$ coincide with equivalent norms, which is the
special case of \cite[Corollary 2.1]{hmy}.

\begin{prop}\label{p2.1}
Let $\eta\in(1, \fz)$ and $p\in(0, \fz)$. The spaces $\rbmop$ and
$\rbmo(\mu)$ coincide with equivalent norms.
\end{prop}

\begin{rem}\label{r2.2}\rm
It was proved in \cite[Lemma 4.6]{h10} that the space $\rbmo(\mu)$
is independent of the choice of $\eta$. By this and Proposition
\ref{p2.1}, it is obvious that the space $\rbmop$ is independent of
the choice of $\eta$.
\end{rem}

We now recall the definition of the atomic Hardy space introduced in
\cite{hyy10}; see also \cite{ad10}.

\begin{defn}\label{d2.5}\rm
Let $\rho\in(1, \fz)$ and $p\in (1,\,\fz]$. A function $b\in
L^1(\mu)$ is called a $(p,\,1)_\lz$-{\it atomic block} if
\begin{enumerate}
\item[(i)] there exists some ball $B$ such that $\supp (b)\subset B$;

\item[(ii)] $\int_\cx b(x)\,d\mu(x)=0$;

\item[(iii)] for $j=1,\,2$, there exist functions $a_j$
supported on balls $B_j\subset B$ and $\lz_j\in\cc$ such that
$$
b=\lz_1a_1+\lz_1a_2,
$$
and
$$
\|a_j\|_{L^p(\mu)}\le[\mu(\rho B_j)]^{1/p-1}[1+\dz(B_j,\,B)]^{-1}.
$$
Moreover, let
$$
|b|_{H_{atb}^{1,\,p}(\mu)}\equiv|\lz_1|+|\lz_2|.
$$
\end{enumerate}

A function $f\in L^1(\mu)$ is said to belong to the {\it atomic
Hardy space} $H_{atb}^{1,\,p}(\mu)$ if there exist
$(p,\,1)_\lz$-atomic blocks $\{b_j\}_{j\in\nn}$ such that
$f=\sum_{j=1}^\fz b_j$ and $\sum_{j=1}^\fz
|b_j|_{H_{atb}^{1,\,p}(\mu)}<\fz$. The {\it norm} of $f$ in
$H_{atb}^{1,\,p}(\mu)$ is defined by
$$
\|f\|_{H_{atb}^{1,\,p}(\mu)}\equiv \inf
\lf\{\sum_j|b_j|_{H_{atb}^{1,\,p}(\mu)}\r\},
$$
where the infimum is taken over all the possible decompositions of
$f$ as above.
\end{defn}

\begin{rem}\label{r2.3}\rm
It was proved in \cite{hyy10} that for each $p\in(1,\,\fz]$, the
atomic Hardy space $H_{atb}^{1,\,p}(\mu)$ is independent of the
choice of $\rho$, and that for all $p\in(1,\,\fz)$, the spaces
$H_{atb}^{1,\,p}(\mu)$ and $H_{atb}^{1,\,\fz}(\mu)$ coincide with
equivalent norms. Thus, in the following, we denote
$H_{atb}^{1,\,p}(\mu)$ simply by $H^1(\mu)$.
\end{rem}

At the end of this section, we recall the $(\az,\,\bz)$-doubling
property of some balls and the Calder\'on-Zygmund decomposition
established by Anh and Duong \cite[Theorem 6.3]{ad10}.

Given $\az,\,\bz\in(1,\,\fz)$, a ball $B\subset\cx$ is called {\it
$(\az,\,\bz)$-doubling} if $\mu(\az B)\le\bz\mu(B)$. It was proved
in \cite{h10} that if a metric measure space $(\cx,\,d,\,\mu)$ is
upper doubling and $\bz>C_\lz^{\log_2\az}\equiv\az^\nu$, then for
every ball $B\subset\cx$, there exists some
$j\in\zz_+\equiv\nn\cup\{0\}$ such that $\az^j B$ is
$(\az,\,\bz)$-doubling. Moreover, let $(\cx,\,d)$ be geometrically
doubling, $\bz>\az^n$ with $n\equiv\log_2 N_0$ and $\mu$ a Borel
measure on $\cx$ which is finite on bounded sets. Hyt\"onen
\cite{h10} also showed that for $\mu$-almost every $x\in\cx$, there
exist arbitrarily small $(\az,\,\bz)$-doubling balls centered at
$x$. Furthermore, the radius of these balls may be chosen to be of
the form $\az^{-j}r$ for $j\in\nn$ and any preassigned number
$r\in(0,\,\fz)$. Throughout this paper, for any $\az\in(1,\,\fz)$
and ball $B$, ${\wz B}^\az$ denotes the {\it smallest
$(\az,\,\bz_\az)$-doubling ball of the form $\az^j B$ with
$j\in\zz_+$}, where
\begin{equation}\label{e2.2}
\bz_\az\equiv\max\lf\{\az^n,\,\az^\nu\r\}+30^n+30^\nu
=\az^{\max\{n,\,\nu\}}+30^n+30^\nu.
\end{equation}

\begin{lem}\label{l2.1}
Let $p\in[1,\,\fz)$, $f\in L^p(\mu)$ and $\ell\in(0,\fz)$
($\ell>\ell_0\equiv \gz_0\|f\|_{L^p(\mu)}/\mu(\cx)$ if
$\mu(\cx)<\fz$, where $\gz_0$ is any fixed positive constant
satisfying that $\gz_0>max\{C_\lz^{3\log_2 6},\,6^{3n}\}$, $C_\lz$
is as in \eqref{e2.2} and $n=\log_2 N_0$). Then

\vspace{0.2cm} {\rm (i)} there exists an almost disjoint family
$\{6B_j\}_j$ of balls such that $\{B_j\}_j$ is pairwise disjoint,
$$
\frac{1}{\mu(6^2B_j)}\int_{B_j}|f(x)|^p\,d\mu(x)
>\frac{\ell^p}{\gz_0} \quad{\text for\, all\, j,}
$$
$$
\frac{1}{\mu(6^2\eta B_j)}\int_{\eta B_j}|f(x)|^p\,d\mu(x) \le
\frac{\ell^p}{\gz_0} \quad{\text for\, all\, j\, and\, all\,
\eta>1,}
$$
and
$$
|f(x)|\le \ell \quad{\text for\, \mu-almost\, every\, x\in
\cx\setminus(\cup_j6B_j);}
$$

{\rm (ii)} for each j, let $S_j$ be a $(3\times
6^2,\,C_\lz^{\log_2(3\times 6^2)+1})$-doubling ball concentric with
$B_j$ satisfying that $r_{S_j}>6^2r_{B_j}$, and
$\oz_j\equiv\chi_{6B_j}/(\sum_k\chi_{6B_k})$. Then there exists a
family $\{\fai_j\}_j$ of functions such that for each $j$,
$\supp(\fai_j)\subset S_j$, $\fai_j$ has a constant sign on $S_j$
and
$$
\int_\cx\fai_j(x)\,d\mu(x)=\int_{6B_j}f(x)\oz_j(x)\,d\mu(x),
$$
$$
\sum_j|\fai_j(x)|\le \gz\ell \quad{\text for\, \mu-almost\, every\,
x\in\cx,}
$$
where $\gz$ is some positive constant depending only on
$(\cx,\,\mu)$, and there exists a positive constant $C$, independent
of $f$, $\ell$ and $j$, such that
$$
\|\fai_j\|_{L^\fz (\mu)}\mu(S_j)\le
C\int_\cx|f(x)\oz_j(x)|\,d\mu(x),
$$
and if $p\in(1,\,\fz)$,
\begin{align*}
&\lf\{\int_{S_j}|\fai_j(x)|^p\,d\mu(x)\r\}^{1/p}[\mu(S_j)]^{1/p'}\le
\frac{C}{\ell^{p-1}}\int_\cx |f(x)\oz_j(x)|^p\,d\mu(x);
\end{align*}

{\rm (iii)} if for any $j$, choosing $S_j$ in {\rm(ii)} to be the
smallest $(3\times6^2,\,C_\lz^{\log_2(3\times 6^2)+1})$-doubling
ball of $(3\times 6^2)B_j$, then $h\equiv\sum_j(f\oz_j-\fai_j)\in
H_{atb}^{1,\,p}(\mu)$ and there exists a positive constant $C$,
independent of $f$ and $\ell$, such that
$$
\|h\|_{H_{atb}^{1,\,p}(\mu)}\le
\frac{C}{\ell^{p-1}}\|f\|_{L^p(\mu)}^p.
$$
\end{lem}

\section{Proof Theorem \ref{t1.1}}\label{s3}

\hskip\parindent To prove Theorem \ref{t1.1}, we also need some
maximal functions in \cite{h10,ad10} as follows. Let $f\in
L_\loc^1(\mu)$ and $x\in\cx$. The {\it doubling Hardy-Littlewood
maximal function} $N(f)(x)$ and the {\it sharp maximal function}
$M^\sharp(f)(x)$ are respectively defined by setting,
$$
N(f)(x)\equiv\sup_{\gfz{B\ni x}{B (6, \,\bz_6)-{\rm
doubling}}}\dfrac{1}{\mu(B)}\int_B|f(y)|\,d\mu(y),
$$
and
\begin{align*}
M^\sharp(f)(x)&\equiv\sup_{B\ni x}
\frac{1}{\mu(5B)}\int_B|f(y)-m_{\wz B^6}(f)|\,d\mu(y)\nonumber\\
&\hs+\sup_{\gfz{x\in B\subset S}{B,\, S (6,\,\bz_6)-{\rm
doubling}}}\dfrac{|m_B(f)-m_S(f)|}{1+\dz(B,\,S)},
\end{align*}
where for any $f\in L_\loc^1(\mu)$ and ball $B$, $m_B(f)$ means its
average over $B$, namely, $m_B(f)\equiv\dfrac{1}{\mu(B)}\int_B
f(x)\,d\mu(x)$. It was showed in \cite[Lemma 2.3]{hlyy} that for any
$p\in[1,\,\fz)$, $Nf$ is bounded from $L^p(\mu)$ to
$L^{p,\,\fz}(\mu)$.

\begin{lem}\label{l3.1}
Let $f\in {\rm RBMO(\mu)}$, $r\in(0, 1)$ and
$M^\sharp_r(f)\equiv[M^\sharp(|f|^r)]^{1/r}$. Then we have
$M^\sharp_rf\in L^\fz(\mu)$, and moreover,
$$\|M^\sharp_rf\|_{L^\fz(\mu)}\ls\|f\|_{{\rm RBMO(\mu)}}.$$
\end{lem}

\begin{proof}
From Remark \ref{r2.2}, we deduce that for any ball $B$,
$$\lf|f_{\wz B^6}-m_{\wz B^6}(f)\r|\le \dfrac1{\mu(\wz B^6)}\dint_{\wz B^6}\lf|f(x)-f_{\wz B^6}\r|\,d\mu(x)\ls \|f\|_{\rbmo(\mu)}.$$
On the other hand, by Proposition \ref{p2.1} and Remark \ref{r2.2}, we see that
$$\|f\|_{\rbmo(\mu)}\sim\|f\|_{{\rm RBMO}_5^r(\mu)}.$$
From these facts, it follows that
\begin{align*}
&\dfrac1{\mu(5B)}\dint_B\lf||f(x)|^r-m_{\wz B^6}(|f|^r)\r|\,d\mu(x)\\
&\hs\le \dfrac1{\mu(5B)}\dint_B\lf[\lf||f(x)|^r-\lf|m_{\wz B^6}(f)\r|^r\r|+\lf|\lf|m_{\wz B^6}(f)\r|^r-m_{\wz B^6}(|f|^r)\r|\r]\,d\mu(x)\\
&\hs\ls \dfrac1{\mu(5B)}\dint_B\lf|f(x)-f_B\r|^r\,d\mu(x)+\lf|f_B-f_{\wz B^6}\r|^r+\lf|f_{\wz B^6}-m_{\wz B^6}(f)\r|^r\\
&\hs\hs+\dfrac1{\mu{(\wz B^6)}}\dint_{\wz B^6}\lf|f(x)-f_{\wz B^6}\r|^r\,d\mu(x)\\
&\hs\ls \lf[1+\dz\lf(B, {\wz B^6}\r)\r]^r\|f\|^r_{{\rm
RBMO}(\mu)}\ls\|f\|^r_{{\rm RBMO}(\mu)},
\end{align*}
where the last inequality follows from the fact that $\dz(B, {\wz
B^6})\ls 1$, which holds by \cite[Lemma 2.1]{hyy10}.

On the other hand, for any $(6, \bz_6)$-doubling balls $B\subset S$,
\begin{align*}
|m_B(|f|^r)-m_S(|f|^r)|&\le \lf|m_B(|f|^r)-|f_B|^r\r|+||f_B|^r-|f_S|^r|+||f_S|^r-m_S(|f|^r)|\\
&\le m_B(|f-f_B|^r)+|f_B-f_S|^r+m_S(|f-f_S|^r)\\
&\ls[1+\dz(B, S)]^r\|f\|^r_{{\rm RBMO}(\mu)}.
\end{align*}
Combining these two inequalities finishes the proof of Lemma
\ref{l3.1}.
\end{proof}

\begin{lem}\label{l3.2}
Let $p\in[1,\,\fz)$ and $f\in L^1_{\rm loc}(\mu)$ such that $\int_\cx f(x)\,d\mu(x)=0$
if $\mu(\cx)<\fz$. If for any $R>0$,
$$
\sup_{0<\ell<R}\ell^p\mu(\{x\in\cx:\,N(f)(x)>\ell\})<\fz,
$$
we then have
\begin{equation*}
\sup_{\ell>0}\ell^p\mu(\{x\in\cx:\,N(f)(x)>\ell\})
\ls\sup_{\ell>0}\ell^p\mu\lf(\lf\{x\in\cx:\,M^\sharp (f)(x)>\ell\r\}\r).
\end{equation*}
\end{lem}

\begin{proof}
Recall the $\lz$-good inequality in \cite{ad10} that
for some fixed constant $\nu\in(0, 1)$ and all $\ez\in(0, \fz)$, there exists some $\dz>0$
such that for any $\ell>0$,
$$\mu\lf(\lf\{x\in\cx:\,N(f)(x)>(1+\ez)\ell,\,M^\sharp(f)(x)\le\dz\ell\r\}\r)
\le\nu\mu(\{x\in\cx:\,N(f)(x)>\ell\}).
$$
From this, it then follows that for $R$ large enough and any $\ez>0$,
\begin{align*}
&\sup_{0<\ell<R}\ell^p\mu(\{x\in\cx:\,N(f)(x)>\ell\})\\
&\hs\le\sup_{0<\ell<R}[(1+\ez)\ell]^p\mu(\{x\in\cx:\,N(f)(x)>(1+\ez)\ell\})\\
&\hs\le\nu(1+\ez)^p\sup_{0<\ell<R}\ell^p\mu(\{x\in\cx:\,N(f)(x)>\ell\})\\
&\hs\hs+(1+\ez)^p\sup_{\ell>0}\ell^p\mu\lf(\lf\{x\in\cx:\,M^\sharp(f)(x)>\dz\ell\r\}\r).
\end{align*}
Choosing $\ez$ small enough such that $\nu(1+\ez)^p<1$, our
assumption then implies that
\begin{align*}
\sup_{0<\ell<R}\ell^p\mu(\{x\in\cx:\,N(f)(x)>\ell\})
\ls\sup_{\ell>0}\ell^p\mu\lf(\lf\{x\in\cx:\,M^\sharp(f)(x)>\ell\r\}\r).
\end{align*}
Letting $R\rightarrow\fz$ then leads to the conclusion, which
completes the proof of Lemma \ref{l3.2}.
\end{proof}

\begin{lem}\label{l3.3}
Let $r\in(0,\,1)$ and $N_r(f)\equiv [N(|f|^r)]^{1/r}$. Then for any
$p\in[1,\,\fz)$, there exists a positive constant $C$, depending on
$r$, such that for suitable function $f$ and any $\ell>0$,
$$
\mu(\{x\in\cx:\,N_r(f)(x)>\ell\})\le C\ell^{-p}
\sup_{\tau\ge\ell}\tau^p\mu(\{x\in\cx:\,|f(x)|>\tau\}).
$$
\end{lem}

\begin{proof}
For each fixed $\ell>0$ and function $f$, decompose $f$ as
$$
f(x)=f(x)\chi_{\{x\in\cx:\,|f(x)|\le\ell\}}(x)
+f(x)\chi_{\{x\in\cx:\,|f(x)|>\ell\}}\equiv f_1(x)+f_2(x).
$$
By the boundedness of $N$ form $L^p(\mu)$ to $L^{p,\,\fz}(\mu)$, we
obtain that
\begin{align*}
\mu(\{x\in\cx:\,N_r(f)(x)>2^{1/r}\ell\})&\le\mu(\{x\in\cx:\,N(|f_2|^r)(x)>\ell^r\})\\
&\ls\ell^{-rp}\int_{\cx}|f_2(x)|^{rp}\,d\mu(x)\\
&\ls\mu(\{x\in\cx:\,|f(x)|>\ell\})\\
&\hs+\ell^{-rp}\int_\ell^\fz
\tau^{rp-1}\mu(\{x\in\cx:\,|f(x)|>\tau\})\,d\tau\\
&\ls\mu(\{x\in\cx:\,|f(x)|>\ell\})\\
&\hs+\ell^{-p} \sup_{\tau>\ell}\tau^p\mu(\{x\in\cx:\,|f(x)|>\tau\}),
\end{align*}
which implies our desired result.
\end{proof}

\begin{proof}[Proof of Theorem \ref{t1.1}]
By the Marcinkiewicz interpolation theorem, we only need to prove
that for all $f\in L^p(\mu)$ with $p\in(1,\,\fz)$ and $\ell>0$,
\begin{equation}\label{e3.1}
\mu(\{x\in\cx:\,|Tf(x)|>\ell\})\ls\ell^{-p}\|f\|^p_{L^p(\mu)}.
\end{equation}
We further consider the following two cases.

{\it Case (i)} $\mu(\cx)=\fz$. Let $L_b^\fz(\mu)$ be the {\it space of bounded functions with
bounded supports} and
$$L^\fz_{b,\,0}(\mu)\equiv\lf\{f\in L_b^\fz(\mu):\,\int_\cx f(x)\,d\mu(x)=0\r\}.$$
Then in this case, $L^\fz_{b,\,0}(\mu)$ is dense in $L^p(\mu)$ for all $p\in(1, \fz)$. Let $r\in(0,\,1)$ and
$N_r(g)\equiv [N(|g|^r)]^{1/r}$ for any $g\in L^r_{\rm loc}(\mu)$. Notice that $|Tf|\le N_r(Tf)$
$\mu$-almost everywhere on $\cx$. Then by a standard density
argument, to prove \eqref{e3.1}, it suffices to prove that for all
$f\in L_{b,\,0}^\fz(\mu)$ and $p\in(1,\,\fz)$,
\begin{equation}\label{e3.2}
\sup_{\ell>0}\ell^p\mu\lf(\lf\{x\in\cx:\,N_r(Tf)(x)>\ell\r\}\r)
\ls\|f\|^p_{L^p(\mu)}.
\end{equation}

For each fixed $\ell>0$, applying Lemma \ref{l2.1}, we obtain that
$f=g+h$, where $h$ is as Lemma \ref{l2.1} and $g\equiv f-h$, such
that
\begin{equation}\label{e3.3}
\|g\|_{L^\fz(\mu)}\ls\ell,\quad h\in H^1(\mu)
\end{equation} and
\begin{equation}\label{e3.4}
\|h\|_{H^1(\mu)}\ls\ell^{1-p}\|f\|^p_{L^p(\mu)}.
\end{equation}
For each $r\in(0,\,1)$, define
$M^\sharp_r(f)\equiv\{M^\sharp(|f|^r)\}^{1/r}$. Then,
\eqref{e3.3} together with the boundedness of $T$ from $L^\fz(\mu)$
to $\rbmo(\mu)$ and Lemma \ref{l3.1} shows that the function
$M^\sharp_r(Tg)$ is bounded by a multiple of $\ell$. Hence, if $c_0$
is a sufficiently large constant, we have
\begin{equation}\label{e3.5}
\lf\{x\in\cx:\,M^\sharp_r(Tg)(x)>c_0\ell\r\}=\emptyset.
\end{equation}
On the other hand, since both $f$ and
$h$ belong to $H^1(\mu)$, we see that $g\in H^1(\mu)$ and
$$\|g\|_{H^1(\mu)}\le \|f\|_{H^1(\mu)}+\|h\|_{H^1(\mu)}\ls \|f\|_{H^1(\mu)}+\ell^{1-p}\|f\|_\lp^p.$$
 By this together with the boundedness of
$T$ from $H^1(\mu)$ to $L^{1,\,\fz}(\mu)$ and Lemma \ref{l3.3}, we
have that for any $p\in(1,\,\fz)$ and $R>0$,
\begin{align*}
\sup_{0<\ell<R}\ell^p\mu\lf(\lf\{x\in\cx:\,N_r(Tg)(x)>\ell\r\}\r)
\ls\sup_{0<\ell<R}\ell^{p-1}
\sup_{\tau\ge\ell}\tau\mu(\{x\in\cx:\,|Tg(x)|>\tau\})<\fz.
\end{align*}
From this, \eqref{e3.5}, Lemma \ref{l3.2} and the fact that
$N_r\circ T$ is quasi-linear, we deduce that there exists a positive
constant $\wz C$ such that
\begin{align}\label{e3.6}
&\sup_{\ell>0}\ell^p\mu\lf(\lf\{x\in\cx:\,N_r(Tf)(x)>\wz Cc_0\ell\r\}\r)\nonumber\\
&\hs\le\sup_{\ell>0}\ell^p\mu\lf(\lf\{x\in\cx:\,N_r(Tg)(x)>c_0\ell\r\}\r)
+\sup_{\ell>0}\ell^p\mu\lf(\lf\{x\in\cx:\,N_r(Th)(x)>c_0\ell\r\}\r)\nonumber\\
&\hs\ls\sup_{\ell>0}\ell^p\mu\lf(\lf\{x\in\cx:\,M^\sharp_r(Tg)(x)>c_0\ell\r\}\r)
+\sup_{\ell>0}\ell^p\mu\lf(\lf\{x\in\cx:\,N_r(Th)(x)>c_0\ell\r\}\r)\nonumber\\
&\hs\ls\sup_{\ell>0}\ell^p\mu\lf(\lf\{x\in\cx:\,N_r(Th)(x)>\ell\r\}\r).
\end{align}
From the boundedness of $N$ from $L^1(\mu)$ to $L^{1,\,\fz}(\mu)$
and the boundedness of $T$ from $H^1(\mu)$ to $L^{1,\,\fz}(\mu)$, it
follows that
\begin{align*}
&\mu(\{x\in\cx:\,N_r(Th)(x)>\ell\})\\
&\hs\le\mu\lf(\lf\{x\in\cx:\,
N\lf(|Th|^r\chi_{\{x\in\cx:\,|(Th)(x)|>\ell/2^{\frac1r}\}}\r)>\frac{\ell^r}{2}\r\}\r)\\
&\hs\ls\ell^{-r}\int_\cx \lf|(Th)(x)
\chi_{\{x\in\cx:\,|(Th)(x)|>\ell/2^{\frac1r}\}}(x)\r|^r\,d\mu(x)\\
&\hs\ls\ell^{-r}\mu\lf(\lf\{x\in\cx:\,|(Th)(x)|>\ell/2^{\frac1r}\r\}\r)
\int_0^{\ell/2^{\frac1r}}s^{r-1}\,ds\\
&\hs\hs+\ell^{-r}\int_{\ell/2^{\frac1r}}^\fz
s^{r-1}\mu(\{x\in\cx:\,|(Th)(x)|>s\})\,ds\\
&\hs\ls\mu\lf(\lf\{x\in\cx:\,|(Th)(x)|>\ell/2^{\frac1r}\r\}\r)
+\dfrac{1}{\ell}\sup_{s\ge\ell/2^{\frac1r}}s\mu(\{x\in\cx:\,|(Th)(x)|>s\})\\
&\hs\ls\dfrac{\|h\|_{H^1(\mu)}}{\ell}\ls\ell^{-p}\|f\|^p_{L^p(\mu)},
\end{align*}
which together with \eqref{e3.6} yields \eqref{e3.2}.

{\it Case (ii)} $\mu(\cx)<\fz$. In this case, assume that $f\in L^\fz_b(\mu).$ Notice that if
$\ell\in(0, \ell_0]$, where $\ell_0$ is as in Lemma \ref{l2.1}, then
\eqref{e3.1} holds trivially. Thus, we only have to consider the
case when $\ell>\ell_0$. Let $r\in(0,\,1)$, $N_r(f)$ be as in Lemma
\ref{l3.3} and $M^\sharp_r$ as in Case (i). For each fixed
$\ell>\ell_0$, applying Lemma \ref{l2.1}, we obtain that $f=g+h$
with $g$ and $h$ satisfying \eqref{e3.3} and \eqref{e3.4}, which
together with the boundedness of $T$ from $L^\fz(\mu)$ to
$\rbmo(\mu)$ and Lemma \ref{l3.1} yields \eqref{e3.5} for
$M^\sharp_r(Tg)$. We now claim that
\begin{equation}\label{e3.7}
{\rm F}\equiv\frac{1}{\mu(\cx)}\dint_\cx|Tg(x)|^r\,d\mu(x)\ls
\ell^r,
\end{equation}
where the constant depends on $\mu(\cx)$ and $r$.
In fact, since $\mu(\cx)<\fz$, we regard $\cx$ as a ball. Then
$g_0\equiv g-\frac1{\mu(\cx)}\int_\cx g(x)\,d\mu(x)\in H^1(\mu)$. On
the other hand, $|T1|^r\in\rbmo(\mu)$ because of the fact that
$T1\in \rbmo(\mu)$ and Lemma \ref{l3.1}. This together with
$\mu(\cx)<\fz$ implies that
$$
\dint_\cx|T1(x)|^r\,d\mu(x)<\fz.
$$
Then by the boundedness of $T$ from $H^1(\mu)$ to $L^{1,\,\fz}(\mu)$
and \eqref{e3.3}, we have
\begin{eqnarray*}
\dint_\cx|Tg(x)|^r\,d\mu(x)
&&\le\dint_\cx\lf\{\lf|Tg_0(x)\r|^r+\lf|T\lf[\frac1{\mu(\cx)}\dint_\cx g(y)\,d\mu(y)\r](x)\r|^r\r\}\,d\mu(x)\\
&&\ls r\dint_0^{\|g_0\|_{H^1(\mu)}/{\mu(\cx)}}t^{r-1}\mu
\lf(\{x\in\cx:\, |Tg_0(x)|>t\}\r)\,dt\\
&&\hs+r\dint_{\|g_0\|_{H^1(\mu)}/{\mu(\cx)}}^\fz t^{r-1}
\mu\lf(\{x\in\cx:\, |Tg_0(x)|>t\}\r)\,dt+\ell^r\\
&&\ls\mu(\cx)\dint_0^{\|g_0\|_{H^1(\mu)}/{\mu(\cx)}}t^{r-1}\,dt+
\|g_0\|_{H^1(\mu)}\dint_{\|g_0\|_{H^1(\mu)}/{\mu(\cx)}}^\fz t^{r-2}\,dt+\ell^r\\
&&\ls [\mu(\cx)]^{1-r}\|g_0\|_{H^1(\mu)}^r+\ell^r \ls \ell^r,
\end{eqnarray*}
which implies \eqref{e3.7}.

Observe that $\int_\cx(|Tg|^r-{\rm F})\,d\mu(x)=0$ and for any
$R>0$,
$$\sup_{0<\ell<R}\ell^p\mu\lf(\lf\{x\in\cx:\,N(|Tg|^r-{\rm F})(x)>\ell\r\}\r)\le R^p\mu(\cx)<\fz.$$
From this together with Lemma \ref{l3.2}, $M^\sharp_r(\rm F)=0$,
\eqref{e3.7} and an argument similar to that used in Case (i), we
conclude that there exists a positive constant $\wz c$ such that
\begin{eqnarray*}
\sup_{\ell>\ell_0}\ell^p\mu\lf(\lf\{x\in\cx:\,N_r(Tf)(x)>\wz cc_0\ell\r\}\r)
&&\le\sup_{\ell>\ell_0}\ell^p\mu\lf(\lf\{x\in\cx:\,N(|Tg|^r-{\rm F})(x)>(c_0\ell)^r\r\}\r)\\
&&\hs+\sup_{\ell>\ell_0}\ell^p\mu\lf(\lf\{x\in\cx:\,N_r(Th)(x)>c_0\ell\r\}\r)\nonumber\\
&&\ls\sup_{\ell>0}\ell^p\mu\lf(\lf\{x\in\cx:\,M^\sharp_r(Tg)(x)>c_0\ell\r\}\r)\\
&&\hs+\sup_{\ell>0}\ell^p\mu\lf(\lf\{x\in\cx:\,N_r(Th)(x)>c_0\ell\r\}\r)\nonumber\\
&&\ls\sup_{\ell>0}\ell^p\mu\lf(\lf\{x\in\cx:\,N_r(Th)(x)>\ell\r\}\r)\ls\|f\|^p_{L^p(\mu)},
\end{eqnarray*}
where in the first inequality we choose $c_0$ large enough such that
${\rm F}\le (c_0\ell)^r$. This completes the proof of Theorem
\ref{t1.1}.

\end{proof}

{\bf Acknowledgments.}
The first author is supported by the Mathematical Tianyuan Youth
Fund (Grant No. 11026120) of National Natural Science Foundation of
China and Chinese Universities Scientific Fund (Grant No.
2011JS043). The second (corresponding) author is supported by
the National Natural Science Foundation
(Grant No. 11101339) of China.

\bigskip

\noindent Haibo Lin

\medskip

\noindent College of Science, China Agricultural University, Beijing
100083, People's Republic of China

\medskip

\noindent{\it E-mail address}: \texttt{haibolincau@126.com}

\bigskip

\noindent Dongyong Yang (Corresponding author)

\medskip

\noindent School of Mathematical Sciences, Xiamen University, Xiamen
361005, People's Republic of China

\medskip

\noindent{\it E-mail address}: \texttt{dyyang@xmu.edu.cn}

\end{document}